\documentclass[a4paper, 11pt, headings = small, abstract]{scrartcl}

\usepackage{etoolbox}
\usepackage{amsmath,amsthm,mathtools}
\usepackage{iftex}
\ifPDFTeX
  \usepackage[utf8]{inputenc}
  \usepackage[theoremfont]{newpxtext}
  \usepackage[vvarbb, upint, bigdelims]{newpxmath}
  \usepackage[scaled=0.95]{inconsolata}
  \usepackage{dsfont}
  \usepackage{bbm}
\else
  \usepackage[no-math]{fontspec}
  \usepackage[warnings-off={mathtools-colon,mathtools-overbracket}]{unicode-math}
\fi

\usepackage[english]{babel}
\usepackage{csquotes,xcolor,url,enumitem,geometry}
\usepackage[citestyle = numeric-comp,bibstyle = numeric, giveninits=true, natbib = true, backend = bibtex,maxbibnames=99, url=false]{biblatex}

\definecolor{myteal}{RGB}{0 123 137}

\usepackage[colorlinks=true, allcolors = myteal]{hyperref}

\theoremstyle{plain}
\newtheorem{theorem}{Theorem}
\newtheorem*{theorem*}{Theorem}

\newtheorem{lemma}[theorem]{Lemma}
\newtheorem{corollary}[theorem]{Corollary}

\theoremstyle{definition}
\newtheorem*{proofnonlattice}{Proof of Theorem~\ref{thm:main} (non-lattice support)}
\newtheorem*{prooflattice}{Proof of Theorem~\ref{thm:main} (lattice support)}
\newtheorem*{proofrw}{Proof of Theorem \ref{thm:main1}}

\theoremstyle{remark}

\setkomafont{sectioning}{\normalcolor\bfseries}
\setkomafont{descriptionlabel}{\normalcolor\bfseries}
\setkomafont{section}{\Large}
\setkomafont{subsection}{\large}
\setkomafont{subsubsection}{\large}
\setkomafont{paragraph}{\large}
\setkomafont{subparagraph}{\large}
\setkomafont{author}{\large}
\setkomafont{date}{\normalsize}

\newcommand\RR{\mathbb{R}}
\newcommand\PP{\mathbb{P}}
\newcommand\EE{\mathbb{E}}
\newcommand\ZZ{\mathbb{Z}}
\newcommand\CC{\mathbb{C}}
\newcommand\iu{\mathrm{i}}
\newcommand\diff[1]{\mathrm{d}#1}
\DeclareMathOperator{\supp}{supp}

\csuse{@ifpackageloaded}{unicode-math}{
  \newcommand\Ind{𝟙}
  \newcommand\Indic[1]{𝟙_{\{#1\}}}
}{
  \newcommand\Ind{\mathbbm{1}}
  \newcommand\Indic[1]{\mathbbm{1}_{\{#1\}}}
}

\csuse{@ifpackageloaded}{unicode-math}{
  \newcommand\bbGamma{ℾ}
}{
  \newcommand{\bbGamma}{{\mathpalette\makebbGamma\relax}}
  \newcommand{\makebbGamma}[2]{%
    \raisebox{\depth}{\scalebox{1}[-1]{$\mathsurround=0pt#1\mathds{L}$}}%
}}

\AtBeginDocument{
  \renewcommand{\Im}{\operatorname{Im}}
  
}

\renewcommand{\tilde}{\widetilde}%
\csuse{@ifpackageloaded}{unicode-math}{
}{
  \newcommand{\overbar}[1]{\mkern 1.5mu\overline{\mkern-1.5mu#1\mkern-1.5mu}\mkern 1.5mu}
}

\let\originalleft\left
\let\originalright\right
\renewcommand{\left}{\mathopen{}\mathclose\bgroup\originalleft}
\renewcommand{\right}{\aftergroup\egroup\originalright}

\newcommand\revised[1]{#1}

\setlist[enumerate,1]{label={(\roman*)}}

\bibliography{lit_overshoot}

\makeatletter
\patchcmd{\@maketitle}{\huge}{\fontsize{18}{20} \selectfont}{}{}
\makeatother

\title{The uniqueness of the Wiener--Hopf factorisation of Lévy processes and random walks}
\author{%
  Leif Döring%
  \thanks{Institute of Mathematics, University of Mannheim,  Germany. Email: \href{mailto:doering@uni-mannheim.de}{doering@uni-mannheim.de}}%
  \and 
  Mladen Savov%
  \thanks{Faculty of Mathematics and Informatics, Sofia University,  Bulgaria. Email: \href{mailto:msavov@fmi.uni-sofia.bg}{msavov@fmi.uni-sofia.bg}} $^{,}$\thanks{Institute of Mathematics and Informatics, Bulgarian Academy of Sciences, Bulgaria. Email: \href{mailto:mladensavov@math.bas.bg}{mladensavov@math.bas.bg}}
  \and 
  Lukas Trottner%
  \thanks{Department of Mathematics, Aarhus University,  Denmark. Email: \href{mailto:trottner@math.au.dk}{trottner@math.au.dk}}
  \and 
  Alexander R.\ Watson%
  \thanks{Department of Statistical Science, University College London, UK. Email: \href{mailto:alexander.watson@ucl.ac.uk}{alexander.watson@ucl.ac.uk}}
}

\begin{document}
\maketitle

\begin{abstract}
  We prove that the spatial Wiener--Hopf factorisation of a
  Lévy process or random walk \revised{without killing} is unique.


\end{abstract}

\section{Introduction}

Let $X$ be a Lévy process \revised{without killing}, and define $\psi\colon\RR\to\CC$ to be its characteristic exponent: 
$\EE e^{\iu z X_t} = \revised{e^{-t \psi(z)}}$.
The spatial Wiener--Hopf factorisation of $\psi$
 \cite[Theorem 6.15(iv)]{kyprianou2014} 
consists of the identity
\begin{equation}\label{e:whf-0}
  \revised{\psi(z)} = a \kappa_+(z)\kappa_-(-z), 
  \qquad z \in \RR,
\end{equation}
where $\kappa_+$ and $\kappa_-$ are the characteristic \revised{exponents} of certain subordinators,
known respectively as the ascending and descending ladder height processes.
The constant $a>0$ is determined by the normalisation of certain local times,
and we take $a=1$ throughout.

In this work, we address the question of the uniqueness of factorisations
of the form \eqref{e:whf-0}. That is,
suppose that
\begin{equation}\label{e:whf}
  \revised{\psi(z)} = \kappa_+(z)\kappa_-(-z) = \kappa'_+(z)\kappa'_-(-z),
  \qquad z \in \RR,
\end{equation}
for some functions $\kappa_+,\kappa_-,\kappa'_+,\kappa'_-$ defined on
$\CC_u = \{ z \in \CC : \Im z \ge 0 \}$ which are the characteristic
exponents of certain subordinators
(that is, \revised{$\kappa_\pm(\iu \cdot)$ and $\kappa'_\pm(\iu \cdot)$} are Bernstein functions).
We will prove the following result.
\begin{theorem}\label{thm:main}
  There exists some $c>0$ such that
  $\kappa_+(z) = c \kappa'_+(z)$ and $\kappa'_-(z) = c\kappa_-(z)$
  for all $z \in \CC_u$.
\end{theorem}
This result has an important probabilistic consequence for the \textit{theory
of friendship} of Lévy processes developed in \citet{vigondiss}.
If $\kappa_\pm$ are the characteristic exponents of two subordinators,
and there exists a Lévy process $X$ with characteristic exponent $\psi$
satisfying \eqref{e:whf-0}, then these subordinators are called \textit{friends}.
In this case we refer to $X$ as the \textit{bonding Lévy process}.
In \citet{vigondiss} necessary
and sufficient conditions are given for friendship in terms of the
characteristics of the subordinators.
The construction of Lévy processes by these means has been the subject of
intense research over the past decade
\cite{kkp2012,KuznetsovPardo2013,KyprianouPardoWatson2014,kyprianou21}.
However, if one wants to use the ladder height processes of a process
constructed via friendship, for example to describe its hitting
distributions \cite{KPW-cens} or its scale functions \cite[\S 9]{kyprianou2014},
then it is essential that the friendship has a probabilistic meaning. This
can be deduced from our result:
\begin{corollary}\label{coro:levy}
  Two friends $H^+$ and $H^-$ are equal in law to the ascending and
  descending ladder height processes (for some scaling of local time) of
  their bonding Lévy process.
\end{corollary}

We emphasise that the difficulty in the results above is that we consider
only the spatial Wiener--Hopf factorisation, and focus on processes
\revised{without killing}.
If one has access to the spatio-temporal Wiener--Hopf factorisation
(that is, if one considers the bivariate ladder processes of $X$), then,
as shown by \citet{CD-trace} and \citet{Kwa-trace},
even knowing just the ascending factor is enough to uniquely specify
the distribution of $X$, and thereby also to determine the descending factor.

Likewise, when $X$ is killed --- or equivalently, when we consider factorisations of
$\revised{q+\psi(z)}$ for some $q>0$ ---
the uniqueness of the factorisation is well-known. 
The traditional proof
proceeds via Liouville's theorem, and can be found, for example,
in \cite[Theorem~1(f)]{kuznetsov2010}.
One first uses the ratios $\kappa_+(z)/\kappa_+'(z)$ 
and $\kappa_-'(z)/\kappa_-(z)$ to define a non-zero entire function $F$.
Taking a continuous version of the logarithm and using asymptotic properties
of the characteristic exponent of a subordinator, one observes that
$\log F$ is sublinear, and therefore constant, which completes the proof.
However, this argument requires a lower bound for each characteristic exponent;
when $X$ is \revised{not killed}, it may be the case that one of these exponents
$\kappa$ has the property that
$\liminf_{\lvert z\rvert \to \infty, z\in\RR} \kappa(z) = 0$.
Examples of such
subordinators are given in \cite[Example~41.23]{sato2013}, and 
we remark that this condition is related to the
\textit{weak non-lattice} property defined in \cite[\S 2.2.2]{PS-bernstein-gamma}.

There is a probabilistic proof of uniqueness, described in \cite[pp.~583--4]{Rog66},
but as this involves factorising the value of $X$
at its lifetime, it too is restricted to the case where $X$ is killed.
The situation is no better for random walks.
For instance, the uniqueness result found in Theorem~12.1.1 or~12.1.2
of \citet{Bor} is restricted to the
setting where killing is present, and when the killing is removed
(putting $\lvert z\rvert=1$ in the notation of \cite{Bor}), a similar
issue appears involving lower bounds on the factors. 

Despite many attempts, we were unable to find a satisfactory proof of
uniqueness when $X$ is \revised{not killed} using either Liouville's theorem or a
probabilistic argument. Instead, we approach the question using the theory of
tempered distributions.
This idea has precedent:
\citet{GK-liouville} have recently made use of distribution theory to prove a
generalised Liouville theorem for Lévy operators.

Our result for Lévy processes immediately resolves the same problem for
random walks. Let $X_1$ be the step of a one-dimensional random walk 
$X = (X_n : n\ge 0)$ started from $X_0=0$,
and assume $X_1$ is not identically zero.
Define
\[
  \tau^+ = \inf\{ n\ge 1: X_n > 0\},
  \quad
  \tau^{-,w} = \inf\{ n \ge 1: X_n \le 0 \}.
\]
We will now define defective random variables $H^+$ and $H^{-,w}$. Adjoin a `cemetery'
state $\Delta$, and extend any function $f\colon \RR \to \CC$ by setting $f(\Delta) = 0$.
Let $H^+$ take the value $X_{\tau^+}$ on the event $\{\tau^+<\infty\}$
and the value $\Delta$ on its complement. Similarly,
let $H^{-,w}$ take the value $X_{\tau^{-,w}}$
on the event $\{\tau^{-,w}<\infty\}$.
These are the first steps of the strict ascending and weak descending ladder
height processes, respectively.
Following \cite[Corollary 12.2.2]{Bor},
the spatial Wiener--Hopf factorisation of this random walk can be expressed
by the identity
\begin{equation*}
  1-\mathbb{E}\big[e^{\mathrm{i}z X_1}\big]
  =\Big(1-\mathbb{E}\big[e^{\mathrm{i}z H^+}\big]\Big)
  \Big(1-\mathbb{E}\big[e^{-\mathrm{i}z H^{-,w}}\big]\Big),\quad z \in\RR,
\end{equation*}
Now, we can show the following result:
\begin{theorem}\label{thm:main1}
  If $V^+$ is a positive defective random variable and
  $V^-$ is a non-negative defective random variable satisfying
  \begin{equation*}
    1-\mathbb{E}\big[e^{\mathrm{i}z X_1}\big]
    = \Big(1-\mathbb{E}\big[e^{\mathrm{i}z V^+}\big]\Big)
    \Big(1-\mathbb{E}\big[e^{-\mathrm{i}z V^-}\big]\Big),
    \quad z \in\RR,
  \end{equation*}
  then $V^+\overset{d}{=}H^+$ and $V^-\overset{d}{=}H^{-,w}$.
\end{theorem} 
Naturally, an analogue of Corollary~\ref{coro:levy} holds for random walks as well.

\section{Proofs}

We begin with some notation.
If $\kappa$ is the characteristic exponent of a subordinator, then it has the
representation
\begin{equation*}
  \revised{-\kappa(z)}
  = - q + \iu d z + \int_{(0,\infty)} (e^{\iu x z} - 1) \, \mu(\diff{x}),
  \qquad \Im z \ge 0,
\end{equation*}
where $q \ge 0$ is the killing rate, $d \ge 0$ the drift and $\mu$
the Lévy measure, satisfying $\int_{(0,\infty)} (1\wedge x) \, \mu(\diff{x})<\infty$.
We adopt similar notation for the characteristics of the other subordinators
in question.

We will use distribution theory, and refer to
\cite{Vla} for background.
We define $\mathcal{D}$ to be the set of complex-valued smooth
functions with compact support, 
$\mathcal{S}$ to be the set of Schwartz functions
(complex-valued
smooth functions with rapidly decaying derivatives) and $\mathcal{S}'$
the set of tempered distributions (continuous linear functionals from
$\mathcal{S}$ to $\CC$.)
The action of $h \in \mathcal{S}'$ on $\phi \in \mathcal{S}$ is
written $\langle h, \phi \rangle$, and it will often be convenient
to write both the distribution and the action with a dummy variable,
i.e., $h(x)$ and $\langle h(x), \phi(x)\rangle$.
In particular, when $h\in\mathcal{S}'$ is a measure, we will often
still write $h(x)$ in places where, as probabilists, we would usually use the infinitesimal
notation $h(\diff{x})$.

There are a number of useful operations on distributions.
The distributional derivative of $h$ is written $Dh$.
The reflection of $h(x)$ is written $h(-x)$, or $h(-\cdot)$ if
we want to omit the dummy variable, and has the meaning
$\langle h(-x), \phi(x)\rangle = \langle h(x), \phi(-x)\rangle$.
Denote by $\mathcal{F}\phi(z) = \int_{\RR} e^{\iu x z} \phi(x)\diff{x}$
the Fourier transform of $\phi \in \mathcal{S}$, and extend this to
$h \in \mathcal{S}'$ by the identity 
$\langle \mathcal{F} h,\phi\rangle = \langle h, \mathcal{F} \phi\rangle$
\cite[\S 6.2]{Vla}.
We also recall the definition of the support of a distribution from \cite[\S 1.5]{Vla},
as follows. We say that $h\in\mathcal{S}'$ vanishes on an open set $G \subset\RR$
if, for all $\phi \in \mathcal{D}$ with support in $G$,
$\langle h,\phi\rangle=0$. The support of $h$ is defined as
$\supp h = \RR \setminus \bigcup\{ G \subset\RR: G \text{ open},\ h \text{ vanishes on } G\}$.

When $\mu$ is the Lévy measure of some subordinator, we define a distribution
$\bbGamma\mu$ by
\[
  \langle \bbGamma \mu, \phi \rangle = 
  \int_{(0,\infty)} \bigl( \phi(x) - \phi(0) \bigr) \mu(\diff{x}),
  \qquad \phi \in \mathcal{S}.
\]
The properties of a Lévy measure imply that $\bbGamma\mu \in \mathcal{S}'$.
Let $G_+$ be the tempered distribution
\[
  G_+ = -q_+ \delta - d_+ D\delta + \bbGamma \mu_+,
\]
\revised{where $\delta$ is the Dirac mass at zero.}
This has the property that $\mathcal{F}G_+ = \revised{-\kappa_+}$.
Let $U_+ = \int_0^\infty \PP(H_t \in \cdot)\, \diff{t}$ be 
the $0$-resolvent measure of the subordinator $H$ with characteristic
exponent $\kappa_+$, and observe that for all $\epsilon>0$, we have
$\mathcal{F}( e^{-\epsilon \cdot} U_+)(z) = \revised{-\frac{1}{\kappa_+(z + \iu\epsilon)}}$
for $\Im z > -\epsilon$.
Analogous quantities are defined for the other
subordinators involved.

The argument required depends on the support of the Lévy process $X$. We say
that $X$ has \emph{lattice support} if for some (and then any) $t>0$,
the support of the random variable $X_t$ is contained in some lattice strictly
contained in $\RR$.
When $X$ has lattice support, we denote by $\eta>0$ the minimal span of the
support of $X_t$; that is, $X_t$ is supported on $\eta \ZZ$ but on no strict sublattice
thereof. When $X$ does not have lattice support, let $\eta=\infty$.

Since $\psi$ has zeroes exactly at the points of $\frac{2\pi}{\eta}\ZZ$
(understanding this set to be  $\{0\}$ when $\eta=\infty$), we can sensibly define
$F\colon \CC\setminus \frac{2\pi}{\eta}\ZZ \to \CC$ as the holomorphic function given by
\begin{equation}\label{eq:holo}
  F(z)
  =
  \begin{cases}
    \dfrac{\kappa_+(z)}{\kappa'_+(z)}, & \Im z \ge 0, \\[2ex]
    \dfrac{\kappa'_-(-z)}{\kappa_-(-z)}, & \Im z \le 0.
  \end{cases}
\end{equation}
Theorem \ref{thm:main} therefore amounts to the assertion that $F \equiv c$ for some constant $c>0$.

We are now in a position to sketch our argument.
Formally it appears that, for $z \in \RR\setminus \frac{2\pi}{\eta}\ZZ$,
\begin{equation}\label{e:sketch}
  \mathcal{F}(G_+ * U_+')(z) = F(z)  = \mathcal{F}\bigl((G_-' * U_-)(-\cdot) \bigr)(z) ,
\end{equation}
where `$*$' is convolution.
An optimistic approach is to say that $G_+*U_+'(x)$ has support
contained in $[0,\infty)$ and $G_-'*U_-(-x)$ has support contained in
$(-\infty,0]$, and that \eqref{e:sketch} implies that these distributions
are equal. $G_+*U_+'(x)$ must therefore have support $\{0\}$, and hence
be equal to $c\delta(x)$, where $\delta$ is the Dirac mass at $0$,
and this gives that $F = c$.
This is close to the argument that we will make, but it is not quite
valid: the hitch is that \eqref{e:sketch}
holds only on the domain of $F$. However, this does tell us that
$G_+*U_+'(x)$ and $G_-'*U_-(-x)$ differ only by a polynomial (in the non-lattice
case) or a series of polynomials weighted by $e^{\iu a x}$ for $a \in\frac{2\pi}{\eta} \ZZ$
(in the lattice case). Showing that this perturbation is actually zero
requires some bounds on the growth rates of the distributions in question,
which are quite delicate in the lattice case.

This proof outline sounds relatively simple, especially in the non-lattice case,
but the devil is in the details, and in particular, we need to find a rigorous interpretation
of the equation \eqref{e:sketch}.
Addressing these difficulties adds some
complexity to the proof, but the essential idea remains the same.

The proof now begins in earnest.
Define the measure 
\begin{equation*}
  h = (\overbar{\mu}_+ + d_+\delta + q_+\Ind_{\RR_+}) * U_+',
\end{equation*}
where $\overbar{\mu}_+(x) = \Indic{x>0}\mu_+(x,\infty)$, $\delta$ is the Dirac
mass at $0$ and $\Ind_{\RR_+}$ is the indicator function of $\RR_+ = [0,\infty)$.
The intuition is that $-h$ is a primitive of $G_+*U_+'$, though we will neither
prove nor use this assertion.  The key ingredient is the following lemma.
\begin{lemma}
  \label{l:tail}
  For any $\epsilon>0$,
  $\int_{[1,\infty)} x^{-(2+\epsilon)} h(\diff{x}) < \infty$.
\end{lemma}
\begin{proof}
  We consider each part of $h$ separately.
  \begin{enumerate}
    \item\label{i:renewal-meas}
      The first part is $\delta * U_+' = U_+'$. 
      The following calculation provides slightly more than is required for
      the integrability in question, and will be used again for the remaining parts.
      When $\kappa_+'$ corresponds to a subordinator \revised{(without killing)},
      \begin{align*}
        \int_{[1,\infty)} x^{-(1+\epsilon)} U_+'(\diff{x})
        &\le
        \sum_{n= 1}^\infty n^{-(1+\epsilon)} U_+'[n,n+1) \\
        &\le \sum_{n=1}^{N-1} n^{-(1+\epsilon)} U_+'[n,n+1)
        + \biggl( \frac{1}{m_+'}+\rho \biggr) \sum_{n= N}^\infty n^{-(1+\epsilon)}
        < \infty,
      \end{align*}
      where $\rho>0$ is arbitrary and $m_+'\in (0,\infty]$,
      and $N$ satisfying the inequality in question
      exists by the renewal theorem \cite[Theorem 5.3.1]{Revuz84}.
      When $\kappa_+'$ corresponds to a killed subordinator, $U_+'$ is a finite measure,
      and the integrability is immediate.

    \item
      Next, we observe that $\Ind_{\RR_+} * U_+'(x) = U_+'[0,x]$, and again
      by the renewal theorem, $U_+'[0,x]$ grows at most linearly in $x$
      as $x\to \infty$, which establishes that
      $\int_1^\infty x^{-(2+\epsilon)} U_+'[0,x] \diff{x} < \infty$.

    \item
      We consider now the calculation
      \begin{align*}
        \int_1^\infty x^{-(2+\epsilon)} \overbar{\mu}_+ * U_+'(x) \diff{x}
        &= \int_1^\infty x^{-(2+\epsilon)}\! \int_{[0,\infty)} \overbar{\mu}_+(x-y) 
        \, U_+'(\diff{y}) \diff{x} \\
        &= \int_{[0,\infty)}\! \int_0^\infty \overbar{\mu}_+(u) (u+y)^{-(2+\epsilon)}
        \Indic{u+y\ge 1} \diff{u}\, U_+'(\diff{y})
        = I_1+I_2+I_3,
      \end{align*}
      where the integrals $I_1$, $I_2$ and $I_3$ are obtained by the restrictions
      to $\{y \le 1\}$, $\{y>1, u\le 1\}$ and $\{y>1, u>1\}$, respectively.

      These terms can be estimated as follows.
      \begin{align*}
        I_1 &= \int_{[0,1]} \int_0^\infty \overbar{\mu}_+(u) (u+y)^{-(2+\epsilon)}
        \Indic{u+y\ge 1} \, \diff{u} \, U_+'(\diff{y}) \\
        &\le
        \int_{[0,1]} \int_1^\infty \overbar{\mu}_+(u) u^{-(2+\epsilon)}
        \, \diff{u} \, U_+'(\diff{y})
        + \int_{[0,1]} \int_0^1 \overbar{\mu}_+(u) \, \diff{u}\, U_+'(\diff{y}) \\
        &\le
        U_+'[0,1] \biggl(
          \overbar{\mu}_+(1) \int_1^\infty u^{-(2+\epsilon)} \, \diff{u}
          + \int_0^1 \overbar{\mu}_+(u) \, \diff{u}
        \biggr)
        <\infty.
      \end{align*}
      Looking at $I_2$, we have
      \begin{align*}
        I_2
        &=
        \int_{(1,\infty)} \int_0^1 \overbar{\mu}_+(u)
        (u+y)^{-(2+\epsilon)} \Indic{u+y\ge 1} \, \diff{u}\, U_+'(\diff{y})
        \\
        &\le \int_{(1,\infty)} y^{-(2+\epsilon)} \, U_+'(\diff{y})
        \int_0^1 \overbar{\mu}_+(u) \, \diff{u},
      \end{align*}
      and this is finite by the argument in part~\ref{i:renewal-meas}.
      Finally we turn to $I_3$.
      \begin{align*}
        I_3
        &=
        \int_{(1,\infty)} \int_1^\infty \overbar{\mu}_+(u)
        (u+y)^{-(2+\epsilon)} \Indic{u+y\ge 1}
        \, \diff{u} \, U_+'(\diff{y}) \\
        &\le \overbar{\mu}_+(1) \int_{(1,\infty)} \int_1^\infty (u+y)^{-(2+\epsilon)}
        \, \diff{u} \, U_+'(\diff{y}) \\
        &= \frac{\overbar{\mu}_+(1)}{2+\epsilon}
        \int_{(1,\infty)} (y+1)^{-(1+\epsilon)} \, U_+'(\diff{y}),
      \end{align*}
      which again is finite by part~\ref{i:renewal-meas},
      and this completes the proof.
      \qedhere
  \end{enumerate}
\end{proof}

An important consequence is that $h$ is a tempered distribution:
\begin{corollary}
  \label{c:tempered}
  $h \in \mathcal{S}'$.
\end{corollary}
\begin{proof}
  A short calculation, similar to the one for $I_1$ in the proof above, reveals
  that $h$ is finite on compact sets.
  The fact that
  $\int \lvert \phi(x)\rvert \Indic{\lvert x \rvert \ge 1}\,  h(\diff{x}) < \infty$,
  for $\phi\in\mathcal{S}$,
  follows from the preceding lemma.
  These two observations imply that $\phi \mapsto \int \phi(x)\, h(\diff{x})<\infty$
  is a map from $\mathcal{S}$ to $\CC$,
  and the fact that it is a continuous operation can be proved using the dominated
  convergence theorem and the lemma.
\end{proof}

With this preparation, we can now give the proof of Theorem \ref{thm:main}.
\begin{proofnonlattice}
  \pushQED{\qed}
  Let $u(z) = \frac{F(z)}{\iu z}$
  for $z = \CC\setminus\{0\}$,
  meaning $u(z) = \frac{\kappa_+(z)}{\iu z \kappa_+'(z)}$ where
  $\Im z \ge 0$;
  note that we do not know whether $u$ is a tempered distribution.
  Our next goal is to show that `$\mathcal{F} h = u$
  away from zero', in a sense that we will make precise.

  Let $\epsilon>0$ and define $h_\epsilon(x) = e^{-\epsilon x}h(x)
  = \bigl(e^{-\epsilon \cdot}( \overbar{\mu}_+ + d_+\delta + q_+\Ind_{\RR_+})\bigr)
  * (e^{-\epsilon \cdot} U_+')(x)$. This is a convolution of finite
  measures, so we can compute its Fourier transform directly as an integral.

  When $\Im z > -\epsilon$, we have (using Fubini's theorem for the integral):
  \begin{align*}
    &\int_{(0,\infty)} e^{\iu x z - \epsilon x} (\overbar{\mu}_+(x) + d_+ \delta(x) + q_+ \Ind_{\RR_+}(x)) \diff{x}\\
    &\,= -\frac{q_+}{\iu(z+\iu \epsilon)} + d_+ 
    + \int_{(0,\infty)} \int_0^x e^{\iu x(z+\iu \epsilon)}  \diff{x}\, \mu_+(\diff{y})
    \\
    &\,= \frac{1}{\iu(z+\iu\epsilon)}
    \biggl(-q_+ + \iu d_+(z+\iu\epsilon)
      + \int_{(0,\infty)} (e^{\iu x(z+\iu\epsilon)} - 1)\,
      \mu_+(\diff{y})
    \biggr)
    \\
    &\,= \revised{-\frac{\kappa_+(z+\iu\epsilon)}{\iu(z+\iu\epsilon)}}.
  \end{align*}
  As already mentioned, for such $z$,
  $\mathcal{F}(e^{-\epsilon \cdot} U_+')(z) = \revised{-\frac{1}{\kappa_+'(z+\iu\epsilon)}}$.
  Using the convolution theorem for the Fourier transform of finite measures,
  we obtain that $\mathcal{F}h_\epsilon(z) = u(z+\iu\epsilon)$ for $\Im z > -\epsilon$.

  Let $\phi\in\mathcal{S}$, and recall that $\mathcal{F}\phi\in \mathcal{S}$
  \cite[Lemma, p.~107]{Vla}. Corollary~\ref{c:tempered} implies that
  $\mathcal{F}\phi(x) h(x)$ is a finite (complex) measure, and so the following
  calculation can proceed via dominated convergence:
  \begin{align*}
    \lim_{\epsilon\to 0} \langle \mathcal{F} h_\epsilon, \phi\rangle
    = \lim_{\epsilon \to 0} \langle h_\epsilon, \mathcal{F}\phi\rangle
    &= \lim_{\epsilon\to 0} \int e^{-\epsilon x} h(x) \mathcal{F}\phi(x)\, \diff{x}
    \\
    &= \int h(x) \mathcal{F}\phi(x) \, \diff{x} 
    = \langle h, \mathcal{F}\phi\rangle
    = \langle \mathcal{F}h, \phi \rangle.
  \end{align*}
  Restricting ourselves now just to $\phi\in\mathcal{D}$ 
  such that $0 \notin\supp\phi$, we obtain:
  \begin{align*}
    \langle \mathcal{F}h , \phi \rangle
    = \lim_{\epsilon\to 0} \langle \mathcal{F} h_\epsilon, \phi \rangle
    &= \lim_{\epsilon\to 0} 
    \int \frac{\kappa_+(z+\iu\epsilon)}{\iu(z+\iu\epsilon)\kappa_+'(z+\iu\epsilon)}
    \phi(z) \, \diff{z} \\
    &= \int u(z) \phi(z) \, \diff{z},
  \end{align*}
  where for the last line, we used bounded convergence, since the domain of
  integration is compact and does not contain $0$,
  and the non-lattice condition for the Lévy process
  implies that $\kappa_+'$ has no zeroes except possibly at $0$.

  Using the properties of Fourier transforms \cite[\S 6.3]{Vla}, this implies
  that for such $\phi$,
  \begin{equation}
    \label{e:FDh}
    \langle -\mathcal{F}(Dh)(z), \phi(z)\rangle
    = \langle \mathcal{F}h(z), \iu z \phi(z) \rangle
    = \int \iu zu(z) \phi(z) \, \diff{z}
    = \int F(z) \phi(z) \, \diff{z}.
  \end{equation}
  Written informally, we have found that `$-\mathcal{F}(Dh) = F$ away from zero'.

  Next, we turn to the other half plane, and define
  \begin{equation*}
    \tilde{h}(x) = - (\overbar{\mu}_-' + d_-'\delta + q_-'\Ind_{\RR_+}) * U_-(-x),
  \end{equation*}
  which by the same argument as in Corollary~\ref{c:tempered} is a tempered
  distribution. Carrying out the same steps as above, we find that
  \begin{equation}
    \label{e:FDh-tilde}
    \langle -\mathcal{F}(D\tilde{h})(z), \phi(z)\rangle
    = \int F(z) \phi(z) \, \diff{z},
  \end{equation}
  for $\phi\in\mathcal{D}$ with support not containing $0$.
  Putting together \eqref{e:FDh} and \eqref{e:FDh-tilde}, we have that
  for such $\phi$,
  \begin{equation*}
    \langle \mathcal{F}(Dh), \phi\rangle
    = \langle \mathcal{F}(D\tilde{h}), \phi \rangle .
  \end{equation*}
  In other words, $\supp \mathcal{F}(Dh - D\tilde{h}) \subset \{0\}$.
  By \cite[\S 2.6]{Vla}, there exist
  $N \ge 0$ (finite by \cite[\S 5.2, Corollary~1]{Vla}) and
  coefficients $(a_n)_{0\le n\le N} \subset\CC$ such that
  \[
    \mathcal{F}(Dh - D\tilde{h})
    = \sum_{n=0}^N a_n D^n \delta,
  \]
  which in turn implies that
  \[
    (Dh - D\tilde{h})(x)
    = \sum_{n=0}^N \frac{a_n}{2\pi} (-\iu x)^n
  \]
  and
  \[
    (h-\tilde{h})(x)
    = a_{-1} + \sum_{n=0}^N \frac{a_n}{2\pi}\frac{1}{n+1} (-i)^n x^{n+1},
  \]
  for some $a_{-1}\in\CC$.
  Lemma~\ref{l:tail} tells us that $a_n = 0$ for all $n\ge 1$, meaning that 
  \[
    Dh - D\tilde{h}
    = \frac{a_0}{2\pi}.
  \]
  Recall that $\supp Dh \subset [0,\infty)$ and $\supp D\tilde{h} \subset (-\infty,0]$.
  To make use of this property, we first split the constant:
  \[
    Dh - \frac{a_0}{2\pi} \Ind_{\RR_+} = D\tilde{h} + \frac{a_0}{2\pi} \Ind_{(-\infty,0)}.
  \]
  Consider a function $\phi \in \mathcal{D}$ with $0\notin \supp\phi$.
  We have:
  \begin{equation*}
    \Bigl\langle Dh - \tfrac{a_0}{2\pi}\Ind_{\RR_+}, \phi\Bigr\rangle
    = \Bigl\langle Dh - \tfrac{a_0}{2\pi}\Ind_{\RR_+}, \phi\Ind_{\RR_+}\Bigr\rangle
    = \Bigl\langle D\tilde{h} + \tfrac{a_0}{2\pi} \Ind_{(-\infty,0)}, \phi\Ind_{\RR_+} \Bigr\rangle
    = 0.
  \end{equation*}
  It follows that
  $\supp \Bigl(Dh - \frac{a_0}{2\pi} \Ind_{\RR_+}\Bigr) \subset \{0\}$, and thence that there exist $M\ge 0$
  finite and $(b_m)_{0\le m\le M} \subset \CC$ such that
  \[
    \frac{a_0}{2\pi}\Ind_{\RR_+} - Dh = \sum_{m=0}^M b_m D^m \delta.
  \]
  But since $h$ is a measure, this simplifies to
  \begin{equation}
    \label{e:Dh-a0b0b1}
    Dh = \frac{a_0}{2\pi} \Ind_{\RR_+} - b_0\delta - b_1 D\delta.
  \end{equation}
  Taking primitives in the above equation \cite[\S 2.2, Theorem]{Vla}, we obtain
  \[ h = f - b_1 \delta, \]
  where $f$ is a distribution coming from an absolutely continuous complex measure.
  Therefore, recalling that $h$ is represented by a (positive) measure,
  $b_1 = -h(\{0\}) \le 0$.
  On the other hand, by the equality of distributions established above,
  \[
  D\tilde{h} + \frac{a_0}{2\pi}\Ind_{(-\infty,0)} = Dh - \frac{a_0}{2\pi}\Ind_{\RR_+} = -b_0\delta - b_1D\delta \]
  also, so
  \[ \tilde{h} = \tilde{f} - b_1\delta, \]
  where again $\tilde{f}$ is a distribution arising from an absolutely continuous
  complex measure.
  Thus, since $-\tilde{h}$ is represented by a positive measure,
  $b_1 = -\tilde{h}(\{0\}) \ge 0$.
  Since $b_1\ge 0$ and $b_1\le 0$, we obtain that $b_1 = 0$.

  \revised{We now take the Fourier transform of \eqref{e:Dh-a0b0b1} with $b_1=0$,
  using \cite[equation (6.12)]{Vla} for the $\Ind_{\RR_+}$ term, to obtain}
  \[
    -\mathcal{F}(Dh)(z)
    = \frac{a_0}{2\pi \iu z} - \frac{a_0}{2} \delta(z) + b_0,
  \]
  \revised{where the distribution $1/z$ is to be understood in the sense of principal
  value.}
  In particular, by taking $\phi\in\mathcal{D}$ 
  with $0\notin\supp\phi$ and using \eqref{e:FDh},
  \[
    \int F(z) \phi(z)\, \diff{z}
    =
    \langle -\mathcal{F}(Dh), \phi \rangle
    =
    \int \biggl(\frac{a_0}{2\pi \iu z} + b_0\biggr) \phi(z) \, \diff{z}.
  \]
  From this, it follows
  \begin{equation}\label{e:F-id}
    F(z) = \frac{a_0}{2\pi \iu z} + b_0,
    \qquad z \in \CC \setminus\{0\},
  \end{equation}
  where we used the identity theorem to extend the equality to $z\notin\RR$. 
  Taking reciprocals in the definition \eqref{eq:holo} of $F$, we can repeat the proof
  thus far to obtain
  \begin{equation}\label{e:F-id1}
    \frac{1}{F(z)} = \frac{a^*_0}{2\pi \iu z} + b^*_0,
    \qquad z \in \CC \setminus\{0\}.
  \end{equation}
  Multiplying \eqref{e:F-id} and \eqref{e:F-id1} we arrive at
  \[
    1=\biggl(\frac{a^*_0}{2\pi \iu z} + b^*_0\biggr)\biggl(\frac{a_0}{2\pi \iu z} + b_0\biggr),
    \qquad z \in \CC \setminus\{0\}.
  \]
  Clearly, then $a_0 = a_0^\ast = 0$. 
  Therefore, $F(z)=b_0$ and hence we have shown that
  $\kappa_+(z) = b_0\kappa_+'(z)$ and $\kappa'_-(z) = b_0\kappa_-(z)$
  for $z \in \CC_u \setminus\{0\}$, and the equation extends to $0$
  by taking limits, since all functions involved are continuous there.
  The fact that $b_0 > 0$ follows from the \revised{positivity} of these functions
  at $\iu x$ with $x>0$.
  This completes the proof in the non-lattice case.
  \popQED
\end{proofnonlattice}

We turn now to the case where $X$ has lattice support, and we will begin by collecting
a few facts.
Let $\eta>0$. The support of $X_t$ is contained in $\eta\ZZ$ for some $t>0$
if and only if it is contained in $\eta\ZZ$ for all $t>0$ \cite[Proposition~24.17]{sato2013},
and in this case we say $X$ has support contained in $\eta \ZZ$.
$X$ has support contained in $\eta\ZZ$ if and only if 
$\psi(2\pi/\eta) = 0$ \cite[Theorem~2.4]{Lukacs},
and if this is the case, then in fact $\psi$ is $2\pi/\eta$-periodic,
so that $\psi(2k\pi/\eta) = 0$ for every $k\in\ZZ$.
Moreover, by considering the Lévy-Itô decomposition we can see that this
holds if and only if $X$ is a compound Poisson process whose (finite) Lévy
measure has support contained in $\eta\ZZ$ \cite[Corollary 24.6]{sato2013}.
We also recall that a Lévy process $X$ is a compound Poisson process if and
only if $\psi$ is bounded on $\RR$ \cite[Corollary~3]{bertoin1996}.

From now on, and without loss of generality, we assume that $X$ has support
that is contained in 
$\ZZ$ and in no sublattice thereof; that is, we assume that $\eta=1$ 
in the discussion above  is maximal.

The key information about the structure of the functions $h$ and $\tilde{h}$ 
is captured in the following lemma. 
When working with lattice support, it will be convenient to use the definition
$\overbar{\mu}_+(x) = \mu_+(x,\infty)\Indic{x\ge 0}$ so that in
particular, $\overbar{\mu}_+(0)$ is equal to the total mass of the measure $\mu_+$,
which as we will shortly see is finite;
we adopt the same convention for the other Lévy measure tails involved.
In comparison with the non-lattice case, this does not change the definition of $h$
as a distribution, so we are free to use the results established previously,
in particular Lemma~\ref{l:tail} and Corollary~\ref{c:tempered}.

\begin{lemma}
  \label{l:lattice-structure}
  The tempered distributions $h$ and $\tilde{h}$ are represented by the functions
  \[
    h(x) = (\overbar{\mu}_+ + q_+\Ind_{\RR_+}) * U_+'(x),
    \qquad
    \tilde{h}(x) = -(\overbar{\mu}'_- + q'_- \Ind_{\RR_+}) * U_-(-x),
  \]
  and their distributional derivatives are given by
  \[
    Dh = \sum_{k\ge 0} \bigl( h(k)-h(k-1)\bigr) \delta_k,
    \qquad
    D\tilde{h} = \sum_{k\le 0} \bigl( \tilde{h}(k) - \tilde{h}(k+1)\bigr) \delta_k,
  \]
  which are series converging in $\mathcal{S}'$.
  Moreover,
  \begin{equation}\label{e:lim-h}
    \lim_{k\to\infty} \bigl( h(k)-h(k-1) \bigr) 
    = 0
    = \lim_{k\to-\infty} \bigl( \tilde{h}(k) - \tilde{h}(k+1) \bigr) .
  \end{equation}
\end{lemma}
\begin{proof}
  Since the Lévy measure $\Pi$ of $X$ is concentrated on $\ZZ\setminus\{0\}$,
  \revised{the characteristic exponent $\psi$ satisfies}
  \begin{equation}\label{e:latPsi}
    \revised{-\psi(z)}=\int_{-\infty}^\infty(e^{\mathrm{i}zx}-1)\, \Pi(\diff x)=\sum_{k\in \mathbb{Z}}a_k(e^{\mathrm{i}zk}-1),
  \end{equation}
  where $a_k = \Pi(\{k\})$ for $k\in \mathbb{Z}\setminus\{0\}$, $a_0=0$
  and $A=\Pi(\mathbb{R})=\sum_{k\in\mathbb{Z}}a_k<\infty$.  
  The equation \eqref{e:latPsi} can be reformulated as 
  \begin{equation*}
    \revised{-\psi(z)}=\mathcal{F}\bigg(\sum_{k\in \mathbb{Z}}a_k\delta_k -A\delta\bigg),
  \end{equation*}
  where $\delta_x$ is a Dirac mass at $x\in\RR$,
  $\delta=\delta_0$ and
  we note that $\sum_{k\in \mathbb{Z}}a_k\delta_k -A\delta\in\mathcal{S}'$. 

  We may assume that $\kappa_\pm$ are the characteristic exponents of ascending
  and descending ladder height processes of $X$ (for some choice of normalisation of
  the local times of $X$ at its supremum and infimum).
  Since the ladder height processes also live on $\ZZ$, it follows immediately
  that $d_\pm = 0$. 
  In addition, $q_+q_- = 0$, since the ladder height processes cannot be killed simultaneously. 
  Hence, we have that
  \begin{align*}
    \revised{-\kappa_\pm(z)}
    =-q_\pm+\int_{(0,\infty)}(e^{\mathrm{i}zx}-1)\, \mu_\pm(\diff x)
    &=\mathcal{F}\bigg(-q_\pm\delta+\sum_{k\geq 1}a^\pm_k\delta_k-A^\pm\delta\bigg) \nonumber \\
    &=\mathcal{F}\bigg(-(q_\pm+A^\pm)\delta+\sum_{k\geq 1}a^\pm_k\delta_k\bigg),
  \end{align*}
  where $a^\pm_k=\mu_\pm(\{k\})$ for $k\geq 1$,
  $A^\pm=\sum_{k\geq 1}a^\pm_k<\infty$, and
  $\sum_{k\geq 1}a^\pm_k\delta_k-(q_\pm+A^\pm)\delta\in\mathcal{S}'$.
  The second factorization of \eqref{e:whf} yields that
  \begin{equation*}
    0=\revised{\psi(2\pi)} = \kappa'_+(2\pi)\kappa'_-(-2\pi)
  \end{equation*}
  and therefore at least one of the two subordinators defined by $\kappa'_\pm$ corresponds to a subordinator \revised{without killing,} living on $\mathbb{Z}_+=\mathbb{Z}\cap [0,\infty)$.   
  Without loss of generality, assume that this applies to $\kappa'_+$,
  and denote the potential measure of its corresponding subordinator by
  $U_+'=\sum_{k\geq 0}u'_k\delta_k$. 

  From the équations amicales inversées, see \cite[Section 5.1]{vigondiss}, we have $\mu^\prime_-(-\cdot) = \Pi \ast U^\prime_+$ on $(-\infty,0)$, which establishes that $\mu'_-$ is supported by $\mathbb{Z}_+$ as well. 
   Next, if $d^\prime_- > 0$ were to hold, then we would have
   $\lvert \kappa^\prime_-(z)\rvert \sim d^\prime_- \lvert z \rvert$ as $\lvert z \rvert \to \infty$. Since $\kappa^\prime_+$ is $2\pi$-periodic and is not identically zero
   by the previous assumption, this contradicts the fact that $\psi(z) = \revised{\kappa^\prime_+(z)\kappa^\prime_-(-z)}$ is bounded over $z\in\RR$.
   Thus, $d^\prime_- = 0$, and we can therefore write
  \begin{equation}\label{e:k-L}
    \revised{-\kappa'_-(z)}=-q_-^\prime+\int_{(0,\infty)}(e^{\mathrm{i}zx}-1)\, \mu'_-(\diff x)=\mathcal{F}\bigg(\sum_{k\geq 1}b^-_k\delta_k -(q'_-+B^-)\delta\bigg),
  \end{equation}
  where $b^-_k=\mu'_-(\{k\})$ for $k\geq 1$ and $B^-=\sum_{k\geq 1}b^-_k<\infty$.
  
  We have shown that the subordinators pertaining to $\kappa_\pm$ and $\kappa^\prime_\pm$ have lattice support. Letting $U_- = \sum_{k\geq 0} u_k \delta_k$ we therefore arrive at
  \begin{equation}\label{e:h1}
    h(x)
= (\overbar{\mu}_+ + q_+\Ind_{\RR_+}) * U_+'(x)
= \sum_{0\leq k\leq x}u'_k(q_+ + \overbar{\mu}_+(x-k))
\Indic{x\ge 0}.
  \end{equation}
  and 
  \begin{equation*}
    \tilde{h}(x) 
    = - (\overbar{\mu}_-'  + q_-'\Ind_{\RR_+}) * U_-(-x)
    = -\sum_{0\leq k\leq -x}u_k\big( q_-'+\overbar{\mu}_-'(-x-k)\big)\Indic{x\le 0}.
  \end{equation*}
  This proves the representations of $h,\tilde{h}$ given in the statement of the
  result.
  Lemma~\ref{l:tail} and Corollary~\ref{c:tempered} remain valid for lattice-valued processes, so
  $h,\tilde{h} \in \mathcal{S}'$.
  Next, we observe that since both $\mu_+,\mu_-'$ are supported on $\mathbb{Z}_+$, the functions $x \mapsto q_+ +\overbar{\mu}_+(x)$ and $x \mapsto q_-'+\overbar{\mu}_-'(x)$ are constant on $(l,l+1]$ for any $l\in \mathbb{Z}_+$. This shows that the distributional derivatives
  $Dh,D\tilde{h}$ are supported by $\mathbb{Z}$. In particular, it is easy to check that
  \begin{equation*}
    Dh=\sum_{k\geq 0} (h(k)-h(k-1))\delta_k, \quad D\tilde{h}=\sum_{k\leq 0} (\tilde{h}(k)-\tilde{h}(k+1))\delta_k,
  \end{equation*}
  noting that $h(-1)=0 = \tilde{h}(1)$.
  We observe that the series above are convergent in $\mathcal{S}'$ thanks to
  Lemma~\ref{l:tail} applied to $h,\tilde{h}$. 

  Next, we show \eqref{e:lim-h}.
  We furnish the proof only for the limit to $+\infty$, the other case being analogous using \eqref{e:k-L}. First note that if $q^\prime_+ > 0$, the measure $U^\prime_+$ is finite and hence $h(k) - h(k-1) \to 0$ as $k \to \infty$ is obtained immediately from \eqref{e:h1} by the dominated convergence theorem. Assume therefore $q^\prime_+ = 0$. From \eqref{e:h1}
  we get that
  \[
  h(k)-h(k-1)
  = \sum_{0\leq l\leq k-1}u'_l\big(\overbar{\mu}_+(k-l)-\overbar{\mu}_+(k-1-l)\big)+u'_k(q_++\overbar{\mu}_+(0)),
  \quad k\ge 0.
\]
  Noting that $\overbar{\mu}_+(k-l)-\overbar{\mu}_+(k-1-l)=-a^+_{k-l}$, we can simplify this to
  \[
    h(k)-h(k-1)
    = -\sum_{0\leq l\leq k-1}u'_la^+_{k-l}+u'_k(q_+ +\overbar{\mu}_+(0)).
  \]
  By the renewal theorem \cite[Theorem 5.3.1]{Revuz84}
  it holds that  $\lim_{k\to\infty}u'_k=1/m_+'$, where $m_+'\in(0,\infty]$ is the expectation of the subordinator \revised{(without killing)} pertaining to $\kappa'_+$ \cite[Proposition 5.3.4 and Theorem 5.3.8]{Revuz84}. Since $\sum_{l\ge 1} a_l^+<\infty$ we get, using the renewal theorem \cite[Proposition~5.3.3]{Revuz84} that
  \[
    \lim_{k\to\infty} \bigl( h(k) - h(k-1)\bigr)
    =-\frac{1}{m_+'}\sum_{l\geq 0}a^+_{l}+\frac{1}{m_+'}(q_++\overbar{\mu}_+(0))= \frac{q_+}{m_+^\prime}.
  \]
  To conclude, we need to verify that, if $q_+'=0$ and $m_+'<\infty$, then $q_+=0$.
  Assume the opposite, that is, that $q_+'=0$, $m_+'<\infty$ and $q_+>0$.
  Then $\revised{\lim_{z \to 0} \kappa^\prime_+(z)/z = -\mathrm{i}m^\prime_+}$, so that \eqref{e:whf} implies 
  the existence of the limit
  \begin{equation}\label{e:exp_bonding}
    \mathrm{i}\lim_{z \to 0} \revised{\psi(z)}/z
    = \revised{m^\prime_+\kappa^\prime_-(0)} 
    =m^\prime_+q^\prime_- \in [0,\infty).
  \end{equation}
  Using this and once more \eqref{e:whf}, $q_+ > 0$ implies that 
  $q_- = 0$ and that $\revised{\lim_{z \to 0} \kappa_-(z)/z}$ is finite. The latter is therefore equal to $\revised{-\mathrm{i}m_-}$, where $m_- \in (0,\infty)$ is the expectation of the subordinator \revised{(without killing)} pertaining to $\kappa_-$. We arrive at:
  \[
    \mathrm{i}\lim_{z \to 0} \revised{\psi(z)}/z
    = \mathrm{i}q_+ \lim_{z \to 0} \revised{\kappa_-(-z)}/z = -q_+m_- < 0,
  \] 
  which contradicts \eqref{e:exp_bonding}. This establishes $q_+ = 0$ in the remaining case, which concludes the proof of \eqref{e:lim-h}.
\end{proof}

With this lemma in place we can tackle the theorem in the case of lattice support.

\begin{prooflattice}
  \pushQED{\qed}
  As in the non-lattice case, we represent
  $\mathcal{F}Dh-\mathcal{F}D\tilde{h}$
  in two different ways.

  On the one hand, Lemma~\ref{l:lattice-structure} directly implies that
  \begin{equation}\label{e:Dh1}
    \mathcal{F}Dh(z)=\sum_{k\geq 0} (h(k)-h(k-1))e^{\mathrm{i}zk},\quad 	\mathcal{F}D\tilde{h}(z)=\sum_{k\leq 0} (\tilde{h}(k)-\tilde{h}(k+1))e^{\mathrm{i}zk}.
  \end{equation}
  For brevity, let us write
  $v_k = h(k)-h(k-1)$ for $k\ge 1$, $v_k = \tilde{h}(k)-\tilde{h}(k+1)$
  for $k\le -1$ and $v_0 = h(0) - \tilde{h}(0)$.
  Thus,
  \begin{equation}\label{e:distr1}
    \mathcal{F}Dh(z)-\mathcal{F}D\tilde{h}(z)=\sum_{k\in\ZZ} v_ke^{\mathrm{i}kz}.
  \end{equation}
  On the other hand, emulating the argument in the non-lattice case, we see that,
  for any $\phi\in \mathcal{D}$ whose support is disjoint from $2\pi\mathbb{Z}$,
  \begin{equation}
    \label{e:Dh-F-lattice}
    \langle -\mathcal{F}(Dh)(z), \phi(z) \rangle
    = \int F(z) \phi(z)\, \diff{z}
  \end{equation}
  and
  \[
    \langle \mathcal{F}(Dh), \phi\rangle
    = \langle \mathcal{F}(D\tilde{h}), \phi \rangle.
  \]
  The latter implies that that the support of
  $\mathcal{F}Dh-\mathcal{F}D\tilde{h}$ is contained in $2\pi\mathbb{Z}$.
  Applying \cite[\S 2.6 and \S 5.2, Corollary 1]{Vla}, we obtain that
  \begin{equation*}
    \mathcal{F}Dh-\mathcal{F}D\tilde{h}
    =\sum_{k\in\ZZ} \sum_{n=0}^{N_k} w_{k,n}D^n\delta_{2\pi k},
  \end{equation*}
  for some $N_k\ge 0$ and weights $w_{k,n}$,
  and our first task is to show that $N_k$ is bounded in $k$. To do this, we appeal
  to \cite[\S 5.4, Corollary]{Vla}: there exists $N$ such that, for all $\epsilon\in(0,1/2)$,
  there exists a function $g\in \mathcal{S}$ supported on $2\pi\ZZ+(-\epsilon,\epsilon)$
  such that
  \[
    \mathcal{F}Dh - \mathcal{F}D\tilde{h} = D^N g.
  \]
  Suppose now that $N_k > N$ for some $k$ and $w_{k,N_k}\ne 0$.
  Take some bump function
  $m \in \mathcal{S}$ with the property that $m(z)=1$ on 
  $[2\pi k-\epsilon,2\pi k+\epsilon]$ and $m(z) = 0$ on $\RR\setminus[2\pi k-r,2\pi k+r]$,
  with $r \in (\epsilon,1/2)$. Then,
  \[
    (\mathcal{F}Dh - \mathcal{F}D\tilde{h})m
    = \sum_{n=0}^{N_k} w_{k,n} D^n \delta_{2\pi k}
    = (D^Ng)m,
  \]
  and in particular,
  \[
    \sum_{n=N+1}^{N_k} w_{k,n} D^n\delta_{2\pi k}
    = D^N \tilde{g},
  \]
  where $\tilde{g} \in \mathcal{S}$ is supported on $(2\pi k-\epsilon,2\pi k+\epsilon)$.
  Taking primitives $N_k$ times, this yields
  \[
    w_{k,N_k}\delta_{2\pi k} = \tilde{g}_1 + P,
  \]
  where $\tilde{g}_1$ is the $(N_k-N)$-th primitive of $\tilde{g}$, and
  $P$ is a polynomial of power at most $N_k$. Since $w_{k,N_k} \ne 0$,
  this is a contradiction.
  Therefore, $N_k\le N$ for every $k\in \ZZ$, and so we obtain
  \begin{equation}\label{e:distr2}
    \mathcal{F}Dh-\mathcal{F}D\tilde{h}=\sum_{k\in\ZZ} \sum_{n=0}^Nw_{k,n}D^n\delta_{2\pi k},
  \end{equation}

  Our goal now is to take our two representations \eqref{e:distr1} and \eqref{e:distr2}
  and use them to show that $Dh-D\tilde{h} = 0$.
  To this end, fix $l\in\mathbb{Z}$; we aim to show that $w_{l,n}=0$ for $0\leq n\leq N$.
  Rearranging \eqref{e:distr1} and \eqref{e:distr2} gives us that (as distributions)
  \begin{equation}\label{e:distr3}
    \sum_{k\in\ZZ} v_ke^{\mathrm{i}kz}=\sum_{k\in\ZZ} \sum_{n=0}^Nw_{k,n}D^n\delta_{2\pi k}(z)
    = \sum_{n=0}^Nw_{l,n}D^n\delta_{2\pi l}(z) + \Upsilon(z),
  \end{equation}
  where $\Upsilon\in\mathcal{S}'$ and its  support does not intersect with $(2\pi(l-1),2\pi(l+1))$. 
  Moreover, Lemma~\ref{l:lattice-structure} tells us that
  \begin{equation}\label{e:lim}
    \lim_{k\to \pm\infty} v_k=0.
  \end{equation}
  Assume that $w_{l,n} \neq 0$ for some $n$ and choose $n_0$ maximally among these. 
  Suppose first that $n_0$ is even.
  Then choose $\phi_\sigma(x)=(\sqrt{2\pi}\sigma)^{-1}e^{-(x-2\pi l)^2/2\sigma^2}\in\mathcal{S}$. Taking into account that
  \begin{equation}\label{e:FT}
    \mathcal{F}\phi_\sigma(z)=e^{\mathrm{i}z2\pi l}e^{-\frac{\sigma^2z^2}{2}}
  \end{equation}
  from   
  \eqref{e:distr3} we derive 
  \begin{equation*}
    \sum_{k\in\ZZ} v_ke^{-\frac{k^2\sigma^2}2} = \sum_{k\in\ZZ} v_ke^{\mathrm{i}k2\pi l}e^{-\frac{k^2\sigma^2}2}=\sum_{n=0}^{n_0}(-1)^n w_{l,n}\phi^{(n)}_\sigma(2\pi l) + \langle \Upsilon, \phi_\sigma\rangle.
  \end{equation*}
  It is easy to check that
  \[\phi^{(n)}_\sigma(2\pi l)=\frac{C_n}{\sigma^{n+1}},\quad |C_n|>0,\]
  when $n$ is even and $\phi^{(n)}_\sigma(l)=0$ when $n$ is odd. Therefore, we get 
  \begin{equation}\label{e:distr3a1}
    \sigma^{n_0+1}\sum_{k\in\ZZ} v_ke^{-\frac{k^2\sigma^2}2}=C_{n_0}w_{l,n_0} 
    + \langle \Upsilon, \sigma^{n_0+1}\phi_\sigma\rangle
    + o(\sigma) \text{, as } \sigma \to 0.
  \end{equation}
  Let us show that
  \begin{equation}\label{e:id4}
    \lim_{\sigma\to 0}\sigma\sum_{k\in\ZZ} v_ke^{-\frac{\sigma^2k^2}{2}}=0
  \end{equation}
  Fix $\epsilon>0$ and, by \eqref{e:lim}, choose $K\ge 1$ large enough such that
  $\lvert v_k\rvert <\epsilon$ when $\lvert k\rvert \ge K$.
  Then,
  \begin{equation*}
    \begin{split}
      \limsup_{\sigma\to 0}\sigma\bigg|\sum_{k\in \ZZ} v_ke^{-\frac{\sigma^2k^2}{2}}\bigg|
      &\leq \epsilon\limsup_{\sigma\to 0}\sigma\sum_{\lvert k\rvert \geq K}e^{-\frac{\sigma^2k^2}{2}}\\
      &\leq \epsilon\limsup_{\sigma\to 0}\sigma
      \int_{-\infty}^{\infty}
      e^{-\frac{\sigma^2x^2}{2}} \Indic{\lvert x\rvert \ge K-1} \diff{x}
      \leq \sqrt{2\pi}\epsilon
    \end{split}
  \end{equation*}
  Given that $\epsilon$ is arbitrary, this proves \eqref{e:id4}. This shows that 
  \begin{equation}\label{e:distr3a2}
    0=w_{l,n_0} + \lim_{\sigma\to 0}\langle \Upsilon, \sigma^{n_0+1}\phi_\sigma\rangle.
  \end{equation}
  Take an infinitely differentiable function $\rho$ such that  $\rho(x)=1$ for $|x-2\pi l|>1/2$ and $\rho(x)=0$ for $|x-2\pi l|<1/4$. Then from the fact that the support of $\Upsilon$ does not intersect with $(2\pi l-1,2\pi l+1)$ we get that
  \[\langle\Upsilon, \sigma^{n_0+1}\phi_\sigma\rangle=\langle\Upsilon, \sigma^{n_0+1}\phi_\sigma\rho\rangle\]
  with $\sigma^{n_0+1}\phi_\sigma\rho\in\mathcal{S}$ for any $\sigma>0$. 
  Clearly we have that $\lim_{\sigma\to 0}\sigma^{n_0+1}\phi_\sigma\rho=0$ in $\mathcal{S}$,
  and we deduce from \eqref{e:distr3a2} that $w_{l,n_0}=0$. 

  Now, if $n_0$ is instead odd, we define
  $\gamma_\sigma(x) = -\sigma^2\phi_\sigma'(x) = (\sqrt{2\pi}\sigma)^{-1}(x-2\pi l)e^{-(x-2\pi l)^2/2\sigma^2}\in\mathcal{S}$,
  and from \eqref{e:FT} we evaluate
  \[ \mathcal{F} \gamma_\sigma(z) = \iu z \sigma^2 \mathcal{F}\phi_\sigma(z)
    = 
    \iu\sigma^2ze^{\iu z2\pi l}e^{-\frac{\sigma^2z^2}{2}}
  \]
  Also, we get easily that 
  $\gamma^{(n)}_\sigma(2\pi l) = -\sigma^2 \phi^{(n+1)}_\sigma(2\pi l) = -C_{n+1}/\sigma^{n}$
  with $\lvert C_{n+1}\rvert > 0$ if $n$ is odd,
  and $\gamma^{(n)}_\sigma(2\pi l)=0$ if $n$ is even.
  We use these expressions as in \eqref{e:distr3a1} to get 
  \begin{equation*}
    \sigma^{n_0+2}\sum_k v_k \mathrm{i}ke^{-\frac{k^2\sigma^2}2}
    = C_{n_0+1}w_{l,n_0} + \langle \Upsilon, \sigma^{n_0}\gamma_\sigma\rangle + o(\sigma),
    \text{ as } \sigma\to 0.
  \end{equation*}
  Let us check that
  \begin{equation}\label{e:id4b}\lim_{\sigma\to 0}\sigma^2\sum_{k\in\ZZ} v_k\mathrm{i}ke^{-\frac{\sigma^2k^2}{2}}=0.
  \end{equation}
  Fix $\epsilon>0$, and from \eqref{e:lim},
  choose $K\ge 2$ large enough that $\lvert v_k\rvert < \epsilon$ for $k\ge K$.
  By \eqref{e:id4b}, we get that
  \begin{align*}
    \limsup_{\sigma\to 0}\sigma^2\bigg|\sum_{k\in\ZZ} v_k \mathrm{i}k e^{-\frac{\sigma^2k^2}{2}}\bigg|
    & \leq \epsilon\limsup_{\sigma\to 0}\sigma^2\sum_{|k|\geq K}\lvert k \rvert e^{-\frac{\sigma^2k^2}{2}}
    \nonumber \\
    &\leq \frac{2K\epsilon}{K-1}\limsup_{\sigma\to 0}\sigma^2\int_{K-1}^{\infty}xe^{-\frac{\sigma^2x^2}{2}} \diff{x}
    \leq 4\epsilon \int_{0}^\infty xe^{-\frac{x^2}{2}} \diff{x}.
  \end{align*}
  Given that $\epsilon$ is arbitrary, this verifies \eqref{e:id4b}. 
  Since necessarily $n_0\geq 1$ and $\gamma_{\sigma}(2\pi l)=0$ we obtain, as in the case where $n_0$ is even,
  that
  \[
    \lim_{\sigma\to 0} \langle \Upsilon, \sigma^{n_0}\gamma_\sigma\rangle=\lim_{\sigma\to 0} \langle \Upsilon, \sigma^{n_0}\gamma_\sigma\rho\rangle=0
  \]
  and thus
  \begin{equation}\label{e:distr3b1}
    \lim_{\sigma\to 0}\sigma^{n_0+2}\sum_{k\in\ZZ} v_k \mathrm{i}ke^{-\frac{k^2\sigma^2}2}=w_{l,n_0}.
  \end{equation}
  From \eqref{e:id4b} and \eqref{e:distr3b1} it holds that $w_{l,n_0}=0$,
  and we conclude that $w_{l,n}=0$ for all $0\leq n\leq N$ and $l\in\mathbb{Z}$. So, from \eqref{e:distr1} and \eqref{e:distr2}
  \[0=	\mathcal{F}Dh(z)-\mathcal{F}D\tilde{h}(z)=\sum_k v_ke^{\mathrm{i}kz}\]
  and we conclude that $v_k=0$ for all $\mathbb{Z}$. Given the definition of $v_k$, this says that
  \[
    h(k)=h(k-1),\quad k\geq 1
    \qquad \text{and}\qquad
    \tilde{h}(k)=\tilde{h}(k+1),\quad k\leq -1.
  \]
  Thus, from \eqref{e:Dh1} we arrive at
  \[ \mathcal{F}Dh = h(0). \]
  Finally,
  \eqref{e:Dh-F-lattice} implies that $F$ is constant on $\CC\setminus 2\pi\ZZ$.
  This concludes the proof.
  \popQED
\end{prooflattice}

Finally, we derive the theorem for random walks from the main result.
\begin{proofrw}
  \pushQED{\qed}
  Let $p_{X_1} = \PP(X_1\ne 0) > 0$, and define
  \[
    \psi(z)
    =
    \revised{1-\EE[e^{\iu z X_1}]}
    = 
    p_{X_1} \revised{\bigl(1-\EE[ e^{\iu z X_1} \mid X_1 \ne 0 ] \bigr)},
  \]
  which is the characteristic exponent of a compound Poisson process with rate $p_{X_1}$
  and Lévy measure $\PP(X_1 \in \cdot \mid X_1 \ne 0)$.
  Similarly and with analogous notation, let
  \[
    \kappa_+(z) = \revised{1-\EE[e^{\iu z H^+}]}
  \]
  and
  \[
    \kappa_-(z)
    = 
    \begin{cases}
      p_{H^{-,w}}\revised{\bigl(1- \EE[e^{- \iu z H^{-,w}} \mid H^{-,w} \ne 0]\bigr)},
      & p_{H^{-,w}} > 0, \\
      \revised{\PP(H^{-,w}=\Delta)},
      & p_{H^{-,w}} = 0 .
    \end{cases}
  \]
  which are both characteristic exponents of (possibly killed) compound Poisson subordinators;
  note that $p_{H^{-,w}} = \PP(H^{-,w} \ne 0) = \PP(H^{-,w} \in (0,\infty) \cup\{\Delta\})$.
  Using this notation, the equation
  \[
    \revised{\psi(z)} = \kappa_+(z)\kappa_-(-z)
  \]
  represents the Wiener--Hopf factorisation of a Lévy process.
  On the other hand, we can define
  \[
    \kappa_+'(z) = \revised{1-\EE[e^{\iu z V^+}]},
  \]
  and
  \[
    \kappa_-'(z) 
    = 
    \begin{cases}
      p_{V^-} \revised{\bigl(1- \EE[e^{- \iu z V^-} \mid V^- \ne 0] \bigr)},
      & p_{V^-} > 0, \\
      \revised{\PP(V^- = \Delta)},
      & p_{V^-} = 0,
    \end{cases}
  \]
  which gives us two more characteristic exponents of compound Poisson subordinators,
  again possibly killed, with
  the property that
  \[
    \revised{\psi(z)} = \kappa_+'(z)\kappa_-'(-z).
  \]
  Theorem~\ref{thm:main} states that there exists $c>0$ such that
  \begin{equation}\label{e:rw:whf-c}
    \kappa_+'(z) = c\kappa_+(z),
    \quad
    \kappa_-(z) = c\kappa_-'(z).
  \end{equation}
  Taking the first of these equalities and rearranging it, we obtain
  \[
    \EE[e^{\iu z V^+}] = 1-c + c\EE[e^{\iu z H^+}]
    = \mathcal{F}\bigl( (1-c)\delta + c\PP(H^+ \in \cdot; H^+\ne\Delta)\bigr)(z).
  \]
  Since neither $V^+$ nor $H^+$ have an atom at zero, it follows that $1-c = 0$,
  that is, $c=1$. In turn, this implies that $H^+ \overset{d}{=} V^+$.
  Turning now to the second equality in \eqref{e:rw:whf-c},
  we first see that $p_{H^{-,w}}=0$ if and only if $p_{V^-} = 0$, in which
  case $H^{-,w} \overset{d}{=} V^-$. If this is not the case, then
  we can write
  \[
    \EE[e^{\iu z H^{-,w}} \mid H^{-,w} \ne 0]
    = 1-\frac{p_{V^-}}{p_{H^{-,w}}} 
    + \frac{p_{V^-}}{p_{H^{-,w}}}\EE[e^{\iu z V^-} \mid V^- \ne 0].
  \]
  By the same logic as above, we obtain that $p_{V^-} = p_{H^{-,w}}$
  and that $V^-$ conditioned on $\{ V^-\ne 0\}$ has the same distribution
  as $H^{-,w}$ conditioned on $\{ H^{-,w} \ne 0 \}$.
  Combining these two observations yields that $H^{-,w} \overset{d}{=} V^-$,
  which completes the proof.
  \popQED
\end{proofrw}

\paragraph*{Acknowledgments} 
The authors wish to thank Vincent Vigon for fruitful discussions and valuable insights during early stages of the project. LT gratefully acknowledges financial support of Carlsberg Foundation Young Researcher Fellowship grant CF20-0640.  M. Savov acknowledges that this study is financed by the European Union - NextGenerationEU, through the National Recovery and Resilience Plan of the Republic of Bulgaria, project No. BG-RRP-2.004-0008.

\begingroup
\setlength{\emergencystretch}{1.5em}
\printbibliography
\endgroup

\end{document}